\documentclass[12pt]{article}
\usepackage{CJK,CJKnumb,CJKspace,dashrule,graphicx}
\usepackage{amsmath, epsfig, cite,colordvi,xcolor,tikz}
\usepackage{amssymb}
\usepackage{amsfonts}

\usepackage{graphicx}

\usepackage{mathrsfs}[12pt]
\usepackage{bm, amscd,  amsfonts, amssymb, cite}
\usepackage{amsmath, epsfig, cite, color, colordvi,amsthm}
\usepackage{extarrows}

\newtheorem{thm}{Theorem}[section]

\newtheorem{conj}[thm]{Conjecture}

\numberwithin{equation}{section}
\allowdisplaybreaks
\numberwithin{equation}{section}

\makeatletter \@addtoreset{equation}{section} \makeatother
\def\.{\hskip.06cm}

\tikzstyle{every node}=[circle,inner sep=1pt,fill=white!60]
\tikzstyle{tn}=[shape=circle, draw, color=black!70]

\begin{document}
\begin{CJK*}{GBK}{song}

\begin{center}
{\large\bf  Proof of a Conjecture of Reiner-Tenner-Yong on

Barely Set-valued Tableaux}
\end{center}

\begin{center}

{\small  Neil J.Y. Fan$^1$, Peter L. Guo$^2$ and Sophie C.C. Sun$^3$}

\vskip 4mm
$^1$Department of Mathematics\\
Sichuan University, Chengdu, Sichuan 610064, P.R. China
\\[3mm]

$^{2,3}$Center for Combinatorics, LPMC\\
Nankai University,
Tianjin 300071,
P.R. China

\vskip 4mm

$^1$fan@scu.edu.cn, $^2$lguo@nankai.edu.cn, $^3$suncongcong@mail.nankai.edu.cn

\end{center}

\begin{abstract}
The notion of a barely set-valued semistandard Young tableau was introduced by Reiner, Tenner and Yong in their study of
the probability distribution of   edges in the Young lattice of partitions.
Given a partition $\lambda$ and a positive integer $k$, let ${\mathrm{BSSYT}}(\lambda,k)$ (respectively,
${\mathrm{SYT}}(\lambda,k)$) denote the set of barely set-valued semistandard Young tableaux (respectively, ordinary semistandard Young tableaux) of shape $\lambda$ with entries in row $i$   not exceeding $k+i$.
In the case when $\lambda$ is a rectangular staircase partition $\delta_d(b^a)$, Reiner, Tenner and Yong conjectured that
$|{\mathrm{BSSYT}}(\lambda,k)|= \frac{kab(d-1)}{(a+b)} |{\mathrm{SYT}}(\lambda,k)|$.
In this paper, we establish
a connection between
 barely set-valued  tableaux   and reverse   plane partitions with  designated corners. We show that for
  any shape
 $\lambda$, the expected jaggedness of a subshape of $\lambda$ under the weak probability distribution can be
 expressed as $\frac{2|{\mathrm{BSSYT}}(\lambda,k)|}
 {k|{\mathrm{SYT}}(\lambda,k)|}$. On the other hand,
 when $\lambda$ is a balanced shape  with $r$
 rows and $c$ columns,
 Chan,  Haddadan,   Hopkins and  Moci proved that the expected jaggedness of a subshape in $\lambda$ under the weak distribution equals $2rc/(r+c)$.
Hence, for a balanced shape $\lambda$ with $r$
 rows and $c$ columns, we establish  the relation  that $|{\mathrm{BSSYT}}(\lambda,k)|=
\frac{krc}{(r+c)}|{\mathrm{SYT}}(\lambda,k)|$. Since
a rectangular staircase shape $\delta_d(b^a)$  is a  balanced shape, we confirm the conjecture of
Reiner, Tenner and Yong.

\end{abstract}

\section{Introduction}

The structure of a set-valued semistandard Young  tableau was  introduced by
Buch \cite{Buch} in his study of the Littlewood-Richardson rule for stable Grothendieck polynomials. A flagged set-valued semistandard Young tableau, defined by  Knutson, Miller and Yong
\cite{Knutson}, is a set-valued semistandard Young tableau such that
each value in row $i$ does not exceed a positive integer $\phi_i$.
The notion of flagged barely set-valued semistandard Young tableaux arose in the work of Reiner, Tenner and Yong \cite{Reiner} on  the probability distribution of the edges in the Young lattice of partitions.

The main objective of this paper is to prove a
conjecture of Reiner, Tenner and Yong \cite{Reiner} concerning the enumeration of barely set-valued  tableaux. A barely set-valued semistandard Young tableau is a set-valued semistandard Young tableau such that
exactly one square is assigned  two integers and each of the remaining squares is occupied by a single integer. For an example, see Figure \ref{fig BSSYT}.

\begin{figure}[h]
\vspace{-0.8cm}  
\setlength{\abovecaptionskip}{0cm}   
\setlength{\belowcaptionskip}{0cm}   
\begin{center}
\begin{picture}(305,100)
\setlength{\unitlength}{1mm}
\put(30,-30){\begin{picture}(-180,60)(0,-25)
\put(15,30){\line(1,0){20}}
\put(15,25){\line(1,0){20}}
\put(15,20){\line(1,0){15}}
\put(15,15){\line(1,0){10}}
\put(15,10){\line(1,0){10}}
\put(15,10){\line(0,1){20}}
\put(20,10){\line(0,1){20}}
\put(25,10){\line(0,1){20}}
\put(30,20){\line(0,1){10}}
\put(35,25){\line(0,1){5}}
\put(16.5,26.5){\small{$1$}}
\put(21.5,26.5){\small{$2$}}
\put(26.5,26.5){\small{$2$}}
\put(31.5,26.5){\small{$3$}}
\put(15.5,21.5){\small{$24$}}
\put(21.5,21.5){\small{$4$}}
\put(26.5,21.5){\small{$4$}}
\put(16.5,16.5){\small{$5$}}
\put(21.5,16.5){\small{$6$}}
\put(16.5,11.5){\small{$7$}}
\put(21.5,11.5){\small{$7$}}
\end{picture}}
\end{picture}
\caption{A barely set-valued semistandard Young tableau.}\label{fig BSSYT}
\end{center}
\end{figure}

For a partition $\lambda$ and a positive integer $k$, let  ${\mathrm{BSSYT}}(\lambda,k)$ (respectively, ${\mathrm{SYT}}(\lambda,k)$) denote the set of barely set-valued semistandard Young tableaux (respectively, ordinary semistandard Young tableaux) of shape $\lambda$ such that the every integer in row $i$ does not exceed  $k+i$.
When $\lambda$ is a rectangular staircase shape $\delta_d(b^a)$, namely, the Young diagram obtained from  the staircase shape  $\delta_d=(d-1,d-2,\ldots,1)$
by replacing each square by an $a\times b$ rectangle,
Reiner, Tenner and Yong \cite{Reiner}
posed the following conjecture.
\begin{conj}[ {Reiner, Tenner and Yong \cite{Reiner}}]\label{Conj-1}
For any positive integers $a,b,d$ and $k$,
\begin{equation}\label{ccc}
|{\mathrm{BSSYT}}(\delta_d(b^a),k)|=\frac{kab(d-1)}{(a+b)}
  |{\mathrm{SYT}}(\delta_d(b^a),k)|.
\end{equation}
\end{conj}

For  $d=2$, Reiner, Tenner and Yong showed that the above conjecture is true
by employing the RSK algorithm as well as Stanley's hook content formula
for semistandard Young tableaux.

In this paper, we give a representation
of a
 barely set-valued tableau of shape $\lambda$ in terms of a reverse plane partition of shape   $\lambda$ together with a designated corner of a subshape of $\lambda$.
This representation enables us to establish
 a connection between the enumeration of
 barely set-valued tableaux of shape $\lambda$
 and  the expected jaggedness of a subshape of $\lambda$ under the weak distribution.
The notion of the jaggedness of a subshape of  a Young diagram was introduced  by Chan,  Haddadan,   Hopkins and  Moci
\cite{Chan-2}.
 More precisely, we
show that
  the expected jaggedness of a subshape of  $\lambda$ under the weak distribution can be expressed as
   \[\frac{2|\mathrm{BSSYT}(\lambda,k)|}
 {k|{\mathrm{SYT}}(\lambda,k)|}.\]

On the other hand, when $\lambda$ is a balanced shape with $r$
 rows and $c$ columns,  Chan,  Haddadan,   Hopkins and  Moci
 \cite{Chan-2} showed that the expected jaggedness of a subshape of  $\lambda$ under the weak distribution equals
 \[\frac{2rc}{r+c}.\]
 Hence, for a balanced shape $\lambda$ with $r$
 rows and $c$ columns, the following relation holds:
  \begin{equation}\label{aopgnj}
 |{\mathrm{BSSYT}}(\lambda,k)|=
\frac{krc}{(r+c)}|{\mathrm{SYT}}(\lambda,k)|.
\end{equation}
Chan,  Haddadan,   Hopkins and  Moci
 \cite{Chan-2} observed that  a
rectangular staircase partition $\delta_d(b^a)$
is a balanced shape. Restricting
 to a rectangular staircase shape $\delta_d(b^a)$,
 \eqref{aopgnj} yields \eqref{ccc}, and this
 leads to a proof of Conjecture \ref{Conj-1}.

Let us proceed with some terminology and notation.
 Let $\lambda=(\lambda_1,\lambda_2,\ldots,\lambda_\ell)$ be a partition of a nonnegative  integer $n$, that is, $\lambda=(\lambda_1,\lambda_2,\ldots,\lambda_\ell)$ is a sequence of nonnegative  integers such that $\lambda_1\geq \lambda_2\geq \cdots \geq \lambda_\ell\geq 0 $ and $\lambda_1+\lambda_2+\cdots+\lambda_\ell=n$.
The Young diagram
of  $\lambda$ is a left-justified array of squares  with $\lambda_i$ squares in  row $i$.
If no confusion arises,
we do not distinguish a partition and its Young diagram.
A Young diagram is also called a shape.
A set-valued semistandard Young tableau of shape  $\lambda$ is an assignment of finite sets of positive integers into  the squares of $\lambda$
such that the sets in each row (respectively, column) are weakly (respectively, strictly) increasing, see Buch \cite{Buch}.
For two sets $A$ and $B$ of positive integers, we write $A\leq B$ if $\max A\leq \min B$
and $A<B$ if $\max A<\min B$. When the set in each square contains a single  integer,
a  set-valued semistandard Young tableau  becomes  an ordinary semistandard Young tableau.
In the case when exactly one square receives two integers and each of the remaining squares receives only one integer, such a set-valued tableau is called a barely set-valued semistandard Young tableau, see Reiner, Tenner and Yong \cite{Reiner}.

Given  a vector $\phi=(\phi_1,\phi_2,\ldots)$ of positive integers, we say that a set-valued semistandard Young tableau is flagged by $\phi$ if every entry in row $i$ cannot exceed $\phi_i$, see Knutson, Miller and Yong \cite{Knutson}.
In particular, when $\phi=(k+1,k+2,\ldots)$, that is, $\phi_i=k+i$, we use $\mathrm{BSSYT}(\lambda,k)$ (respectively, $\mathrm{SYT}(\lambda,k)$) to  represent the set of barely set-valued semistandard
Young tableaux
(respectively, ordinary semistandard Young tableaux)
flagged by $\phi$.

It is worth mentioning that Conjecture \ref{Conj-1} can be reformulated in terms of  the polynomials $FK(w, \ell)$ defined on 0-Hecke
words of length $\ell$ for a  permutation $w$ \cite{Reiner}.
A 0-Hecke word of a permutation $w$ on $\{1,2,\ldots,n\}$ can be constructed
recursively as follows.
As usual, we use $s_i$ ($1\leq i\leq n-1$) to denote
the simple transposition  that swaps $i$ and $i+1$.
 An expression of $w$ as a product of simple
  transpositions is called reduced if it
 consists of a minimum number of simple
 transpositions.  The length of $w$, denoted $\ell(w)$,
is the number  of simple transpositions in a reduced
expression of $w$.
The  length $\ell(w)$ of $w$
can also be  interpreted
as the number of inversions of $w=w_1w_2\cdots w_n$, that is,
\[\ell(w)=|\{(i,j)\,|\, 1\leq i<j\leq n, w_i>w_j\}|.\]

Given a sequence $S=(s_{i_1},s_{i_2},\ldots,s_{i_\ell})$
of simple transpositions, we construct a permutation, denoted $H(S)$,
by a recursive procedure as follows. If $\ell=1$, set
$H(S)=s_{i_1}$. If $\ell>1$,  let $S'=(s_{i_1},s_{i_2},\ldots,s_{i_{\ell-1}})$ and set
\[
H(S)=\left\{\begin{array}{ll}
H(S'), &\mbox{if $\ell(H(S')\cdot s_{i_\ell})<\ell(H(S'))$},\\[5pt]
H(S')\cdot s_{i_\ell}, &\mbox{if $\ell(H(S')\cdot s_{i_\ell})>\ell(H(S'))$}.\end{array}\right.
\]
If $H(S)=w$, then $S$ is called a 0-Hecke word of $w$ of length $\ell$. It is easily seen that
a 0-Hecke word of $w$ of length $\ell(w)$  is a reduced expression  of $w$.

The polynomials $FK(w,\ell)$ are a generalization of  the following
polynomials  defined by
Fomin and Kirilov \cite{Fomin}:
\begin{equation}\label{EX-1}
\sum_{(s_{i_1},s_{i_2},\ldots,s_{i_{\ell_0}})}(x+i_1)(x+i_2)\cdots (x+i_{\ell_0}),
\end{equation}
where the sum  ranges over reduced
expressions  of the longest permutation $w_0=n (n-1)\cdots 1$ of length $\ell_0=\ell(w_0)=n(n-1)/2$.
Using a counting formula for  the monomials in
a Schubert polynomial due to Macdonald \cite{Macdonald}
and a formula on the number of reverse plane partitions of a staircase shape found by
Proctor \cite{Proctor}, Fomin and Kirilov \cite{Fomin}
established the following relation:
\begin{equation}\label{XY}
\sum_{(s_{i_1},s_{i_2},\ldots,s_{i_{\ell_0}})}(x+i_1)(x+i_2)\cdots (x+i_{\ell_0})=
{n\choose 2}!\prod_{1\leq i< j\leq n} \frac{x+i+j-1}{i+j-1}.
\end{equation}
Equating the leading  coefficients on both sides, \eqref{XY}
gives the number of reduced expressions  of the longest permutation $w_0$, as proved by Stanley \cite{Stanley1}.
Reiner, Tenner and Yong \cite{Reiner} defined the polynomial
$FK(w,\ell)$ as follows:
\[FK(w,\ell)=\sum_{(s_{i_1},s_{i_2},\ldots,s_{i_\ell})}(x+i_1)(x+i_2)\cdots (x+i_\ell),\]
where the sum ranges over 0-Hecke
words of $w$ of length $\ell$. For the case  $w=w_0$ and $\ell=\ell(w_0)$,
$FK(w,\ell)$ reduces to
the polynomial in  \eqref{EX-1}.

Reiner, Tenner and Yong~\cite{Reiner} showed that  Conjecture
\ref{Conj-1} is equivalent to a relation on $FK(w,\ell)$,
where $w$ is  a   dominant permutation whose Lehmer code is
a rectangular staircase shape.
Recall that $w$ is called a dominant permutation if it is 132-avoiding, that is, if there are no
indices $i_1<i_2<i_3$ such that $w_{i_1} < w_{i_3} < w_{i_2}$.
There is  an alternative  characterization of a dominant permutation, that is,   $w$ is  dominant  if and only if the Lehmer code
$(c_1(w),c_{2}(w),\ldots,c_{n}(w))$ of $w$ is a nonincreasing
sequence,
where, for $1\leq i\leq n$, $c_i(w)$ is the number of inversions of $w$ at position $i$,
namely,
\[c_i(w)=|\{j\,|\, i<j, w_i>w_j\}|.\]

Employing properties of Grothendieck polynomials,
Reiner, Tenner and Yong~\cite{Reiner} showed that Conjecture \ref{Conj-1}
is equivalent to the following assertion.

\begin{conj}[{Reiner, Tenner and Yong \cite{Reiner}}]\label{Conj-1-e}
Let $w$ be a dominant permutation whose Lehmer code is a rectangular staircase shape $\lambda=\delta_d(b^a)$. Then
\begin{equation}\label{TTT}
\frac{FK(w,\ell(w)+1)}{FK(w,\ell(w))}
= \binom{ \ell(w)+1 }{2}\left( \frac{ 4x }{d(a+b)} + 1\right).
\end{equation}
\end{conj}

 In this paper, we obtain an extension of   \eqref{TTT}  to a dominant  permutation whose Lehmer code is a balanced shape.

\section{The formula of Chan-Haddadan-Hopkins-Moci}\label{section2}

In this section, we shall give an overview of a formula
of Chan, Haddadan, Hopkins and Moci \cite{Chan-2}
for the expected jaggedness
of a subshape in a Young diagram under a toggle-symmetric distribution. This formula is needed in our proof of the
conjecture of Reiner, Tenner and Yong. It is a
far reaching generalization of a formula derived by
Chan,  L\'opez Mart\'in, Pflueger and  Teixidor i Bigas \cite{Chan-1}.
While the formula is quite involved for a general shape $\lambda$,
as far as this paper is concerned, we only need
the case when  $
 \lambda$ is
a balanced shape. In this case it admits
a closed form. Moreover, we shall restrict our
 attention to a special distribution, namely,  the
weak distribution on the subshapes of a balanced shape. This is possible because as shown in \cite{Chan-2},
the weak distribution is indeed a toggle-symmetric distribution.

Let us begin with the necessary terminology.
A toggle-symmetric probability distribution is a probability distribution on the order ideals of a finite poset subject to
certain symmetry conditions. Given a finite poset $(P,\leq)$,
an order ideal $I$ of $P$ is a subset of $P$
such that if $p\in I$ and  $q\in P$ with $q\leq p$,
then  $q\in I$. Let $J(P)$ denote the  set of order ideals of $P$. We say that an element $p\in  P$ can be toggled into $I$ if $p$ is a minimal element not in $I$, and that
$p$ can be toggled out of $I$ if $p$ is a maximal element  in $I$. To be more specific, we say that
$p$ can be toggled into $I$ if $p\not\in I$
and $I\cup \{p\}$ is an order ideal of $P$, and that
$p$ can be toggled out of  $I$ if $p\in I$
and $I\setminus \{p\}$ is an order ideal of $P$.
For each $p\in P$,  the indicator
random variables $\mathcal{T}_p^+$ and $\mathcal{T}_p^-$ on $J(P)$ are defined as follows. For an order ideal $I$ of $P$, set
$\mathcal{T}_p^+(I)=1$ if $p$ can be toggled into $I$, and
$\mathcal{T}_p^+(I)=0$ otherwise. Similarly,
set
$\mathcal{T}_p^-(I)=1$ if $p$ can be toggled out of $I$, and
$\mathcal{T}_p^-(I)=0$ otherwise.

Given a probability distribution on  $J(P)$ and
an element $p\in P$, the distribution
 is called toggle-symmetric at  $p$
if the probability that $p$
can be toggled into an order ideal $I$ equals the
probability that $p$
can be toggled out of an order ideal $I$. We say that
a distribution on $J(P)$ is toggle-symmetric if it
is toggle-symmetric at every $p\in P$.
In other words, a probability distribution on $J(P)$ is
toggle-symmetric if for every $p\in P$,
the expected value of the random variable $\mathcal{T}_p^+$
equals the expected value of the random variable $\mathcal{T}_p^-$.

When $P$
is the poset corresponding to a skew Young diagram, Chan, Haddadan, Hopkins and Moci \cite{Chan-2}
found a formula for the expected jaggedness of an order ideal of   $P$ for a toggle-symmetric distribution. The jaggedness of an order ideal $I$ of a poset $P$, denoted $\mathrm{jag}(I)$, is defined to be  the total number of elements in $ P$ which can be toggled into $I$ or toggled out of $I$.
In this paper, we shall be concerned only with the posets
corresponding to  Young diagrams.

To a Young diagram $\lambda$, one can associate
a poset structure on the squares of $\lambda$.
For two  squares $B$ and $B'$ of $\lambda$, we say that
$B$ is less than or equal to $B'$ if   $B$ occurs northwest of $B'$. More precisely, assume that $B$ is in row $i$ and column $j$, and $B'$ is in row $i'$ and column $j'$. Then $B\leq B'$ if and only if $i\leq i'$ and $j\leq j'$.
It is readily  seen that a subset of squares of $\lambda$ forms an order ideal with respect to the
above poset structure if and only if it is a subshape of $\lambda$, namely,  a Young diagram contained in $\lambda$.

By the definition of the jaggedness of an order ideal, it is easily seen  that the jaggedness of a subshape $\mu$ of $\lambda$ equals the total number of corners and proper   outside corners of $\mu$ \cite{Chan-2}.
A corner of  a shape $\mu$ is a square in $\mu$ such that the squares immediately below and to the right are not in $\mu$. While
an outside corner of $\mu$ is a square out of $\mu$ such
that the squares immediately above and to the left are in $\mu$, see, for example, the survey of Pak \cite{Pak}.
We assume that the square just to the right of the first
row  and the  square just below the first column
are also outside corners. It should be noticed that the outside corners of $\mu$ are also called   outer boxes of $\mu$, see Stanley \cite[Chapter 7, Appendix 1]{Stanley}. By a
 proper outside corner of $\mu$ we mean an outside corner of $\mu$
 contained in $\lambda$.

Clearly,
a  square of $\lambda$  can be toggled out of $\mu$ if  and only if it is a corner  of $\mu$, while a  square of $\lambda$  can be toggled into $\mu$ if  and only if it is an outside corner  of $\mu$. Thus the jaggedness $\mathrm{jag}(\mu)$
equals the total number of  corners and proper outside
  corners of $\mu$.
For example, the jaggedness of the subshape $(3,3,2,1)$ of the diagram $(4,4,3,2)$ in Figure \ref{PPO}
equals 5, since it has  three corners and two
proper outside corners, which are depicted by  solid
squares and  open   squares respectively.
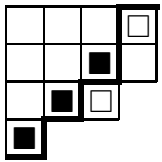
\begin{figure}[h]
\setlength{\abovecaptionskip}{-0.4cm}   
\setlength{\belowcaptionskip}{0cm}   
\begin{center}
\begin{picture}(-120,100)
\setlength{\unitlength}{1mm}
\put(-30,30){\line(1,0){20}}
\put(-30,25){\line(1,0){20}}
\put(-30,20){\line(1,0){20}}
\put(-30,15){\line(1,0){15}}
\put(-30,10){\line(1,0){5}}
\put(-30,10){\line(0,1){20}}
\put(-25,10){\line(0,1){20}}
\put(-20,15){\line(0,1){15}}
\put(-15,15){\line(0,1){15}}
\put(-10,20){\line(0,1){10}}
\put(-19.15,21){$\blacksquare$}
\put(-24.15,16){$\blacksquare$}
\put(-29.15,11){$\blacksquare$}
\linethickness{1.7pt}
\put(-30,10){\line(1,0){5.3}}
\put(-25,10){\line(0,1){5.3}}
\put(-25,15){\line(1,0){5.3}}
\put(-20,15){\line(0,1){5.3}}
\put(-20,20){\line(1,0){5.3}}
\put(-15,20){\line(0,1){10.3}}
\put(-15,30){\line(1,0){5.3}}
\put(-14,26){$\square$}
\put(-19,16){$\square$}
\end{picture}
\caption{ Corners and proper outside corners.}\label{PPO}
\end{center}
\end{figure}

The jaggedness of a subshape $\mu$ of $\lambda$ can also  be described in terms of the total number of left turns and right turns of the lattice path corresponding to $\mu$. By a lattice path in $\lambda$ we mean a lattice path  in $\lambda$ from the bottom left corner to the top right corner consisting of unit east steps and unit north steps.
Clearly, a subshape of $\lambda$ is determined by a lattice path in the Young diagram of $\lambda$.
For example, the thick line in Figure \ref{fig lattice path} is a lattice path in $(4,4,3,1)$ corresponding to the  subshape   $(3,3,2,1)$.
\begin{figure}[h]
\setlength{\abovecaptionskip}{-0.4cm}   
\setlength{\belowcaptionskip}{0cm}   
\begin{center}
\begin{picture}(-115,100)
\setlength{\unitlength}{1mm}
\linethickness{0.25pt}
\put(-30,30){\line(1,0){20}}
\put(-30,25){\line(1,0){20}}
\put(-30,20){\line(1,0){20}}
\put(-30,15){\line(1,0){15}}
\put(-30,10){\line(1,0){5}}
\put(-30,10){\line(0,1){20}}
\put(-25,10){\line(0,1){20}}
\put(-20,15){\line(0,1){15}}
\put(-15,15){\line(0,1){15}}
\put(-10,20){\line(0,1){10}}
\thicklines
\linethickness{1.7pt}
\put(-30,10){\line(1,0){5.3}}
\put(-25,10){\line(0,1){5.3}}
\put(-25,15){\line(1,0){5.3}}
\put(-20,15){\line(0,1){5.3}}
\put(-20,20){\line(1,0){5.3}}
\put(-15,20){\line(0,1){10.3}}
\put(-15,30){\line(1,0){5.3}}
\put(-20,20){\circle{2}}
\put(-15,30){\circle{2}}
\put(-25,10){\circle*{2}}
\put(-20,15){\circle*{2}}
\put(-15,20){\circle*{2}}
\end{picture}
\caption{ Left turns and right turns of a lattice path in a Young diagram.}\label{fig lattice path}
\end{center}
\end{figure}
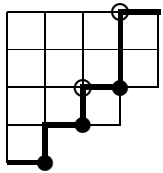

The  notion of   left turns and   right turns  of a lattice path in $\lambda$  was introduced by Chan,  L\'opez Mart\'in, Pflueger and  Teixidor i Bigas \cite{Chan-1} for the computation of the genera of the Brill-Noether curves.
To be more specific, a left turn of a lattice path  in $\lambda$ is an east step immediately  followed by a north step, and a right turn is a north step immediately  followed by
an east step with the additional requirement that
these two consecutive steps are borders of a square of $\lambda$.
In Figure \ref{fig lattice path}, the solid circles and the open circles   represent the left turns and
the right turns of the path in $(4,4,3,1)$, respectively. It is evident that for a subshape $\mu$ of $\lambda$, a corner (respectively, a proper outside corner) of $\mu$ corresponds to a left turn (respectively, a right turn) of the lattice path determined by $\mu$. It follows that the jaggedness   of a subshape  of $\lambda$ also equals the total number of left turns and right turns of the corresponding lattice path in $\lambda$, see \cite{Chan-2}.

Chan, Haddadan, Hopkins and Moci \cite{Chan-2} found a formula for the expected jaggedness of a subshape for a general
skew  Young diagram, which
 turns out to have a closed form when it
is a balanced Young diagram.
A balanced shape is defined
in terms of the positions of  outward corners of a Young diagram $\lambda$. An outward corner of  $\lambda$ is
a north step immediately followed by
an east step along the southeast boundary of $\lambda$.
A Young diagram $\lambda$ is called a balanced shape if the turning point  of each outward corner
of $\lambda$ lies on the main anti-diagonal of $\lambda$, that is, the straight line connecting the
starting point and the terminating point of
a lattice path in $\lambda$.
For example,  Figure \ref{balanced} illustrates two balanced shapes,  where the solid dots stand for  the tuning points of outward corners and
the dashed lines represent the main anti-diagonals.
From the largest part and the number of parts of $\lambda$,
it is easy to determine whether the turning point of an outward corner lies on the anti-diagonal.
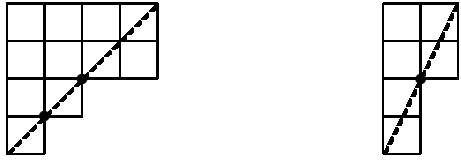
\begin{figure}[h]
\setlength{\abovecaptionskip}{-0.4cm}   
\setlength{\belowcaptionskip}{0cm}   
\begin{center}
\begin{picture}(-5,100)
\setlength{\unitlength}{1mm}
\put(-30,30){\line(1,0){20}}
\put(-30,25){\line(1,0){20}}
\put(-30,20){\line(1,0){20}}
\put(-30,15){\line(1,0){10}}
\put(-30,10){\line(1,0){5}}
\put(-30,10){\line(0,1){20}}
\put(-25,10){\line(0,1){20}}
\put(-20,15){\line(0,1){15}}
\put(-15,20){\line(0,1){10}}
\put(-10,20){\line(0,1){10}}

\put(20,30){\line(1,0){10}}
\put(20,25){\line(1,0){10}}
\put(20,20){\line(1,0){10}}
\put(20,15){\line(1,0){5}}
\put(20,10){\line(1,0){5}}
\put(20,10){\line(0,1){20}}
\put(25,10){\line(0,1){20}}
\put(30,20){\line(0,1){10}}

\thicklines

\put(-25,15){\circle*{1.5}}
\put(-20,20){\circle*{1.5}}
\put(25,20){\circle*{1.5}}

\put(-30,10){\rotatebox{45}{\hdashrule{2.9cm}{1pt}{3pt 1pt}}}
\put(20,10){\rotatebox{64}{\hdashrule{2.3cm}{1pt}{3pt 1pt}}}
\end{picture}
\caption{Balanced Young diagrams.}\label{balanced}
\end{center}
\end{figure}

When $\lambda$ is a balanced shape, Chan, Haddadan, Hopkins and Moci  \cite[Corollary 3.8]{Chan-2} obtained   the following formula for any toggle-symmetric
distribution.

\begin{thm}\label{JD}
For a balanced Young diagram $\lambda$  with $r$ rows  and $c$ columns and for any toggle-symmetric
distribution, the expected jaggedness of a subshape of $\lambda$ equals
\begin{equation*}
\frac{2rc}{r+c}.
\end{equation*}
\end{thm}

We conclude this section with a description of a specific
toggle-symmetric distribution, called  the weak distribution,
see \cite[Definition 2.2]{Chan-2}, which is closely related to
the enumeration of barely set-valued tableaux.
The weak distribution  is defined on reverse plane partitions.
Recall that a reverse plane partition of shape $\lambda$ is
 an assignment of nonnegative integers into the squares  of $\lambda$
such that the integers in each row and each column are weakly increasing, see Stanley \cite[Chapter 7]{Stanley}.
Given a positive integer $k$,
let $\mathrm{RPP}(\lambda, k)$ denote the set of reverse plane partitions of shape $\lambda$ with every entry not exceeding $k$.

To define the weak distribution, consider
the pairs
$(P,i)$ with $P\in \mathrm{RPP}(\lambda, k)$ and $i\in \{1,2,\ldots,k\}$. A pair $(P,i)$  determines a subshape of $\lambda$,  denoted $\alpha(P,i)$, which consists of squares of $P$ occupied by the entries   strictly less than $i$. The subshape $\alpha(P,i)$ is also called an induced subshape.
Let
\[Q(\lambda,k)=\{(P,i)\,|\, P\in \mathrm{RPP}(\lambda, k), 1 \leq i\leq k\}.\]
Assume that the pairs $(P, i)$ in $Q(\lambda, k)$ are
generated uniformly. Then we are led to a distribution of
subshapes of $\lambda$. More precisely, among all the
subshapes $\alpha(P,i)$ generated by the pairs in $Q(
\lambda, k)$,
a subshape $\mu$ occurs with probability
\begin{align}\label{PPB}
\frac{|\{(P,i)\in Q(\lambda, k)\,|\,\alpha(P,i)=\mu\}|}{|Q(\lambda, k)|}.
\end{align}
The distribution defined in \eqref{PPB} is called the weak
distribution on the set of subshapes of $\lambda$.

Chan, Haddadan, Hopkins and Moci \cite[Lemma 2.8]{Chan-2} showed that the weak distribution is indeed a toggle-symmetric
distribution. Hence, in the case when $\lambda$ is a balanced shape, the expected jaggedness    under the weak
distribution can be computed by the formula in Theorem \ref{JD}, and so the following relation holds.

\begin{thm}\label{JD1}
For  a balanced shape $\lambda$ with $r$
rows and $c$ columns,
we have
\begin{equation*}
\frac{\sum_{\mu}|\{(P,i)\in Q(\lambda, k)\,|\,\alpha(P,i)=\mu\}|\mathrm{jag}(\mu)}{|Q(\lambda, k)|}=\frac{2rc}{r+c},
\end{equation*}
where $\mu$ ranges over the subshapes of $\lambda$.
\end{thm}

\section{Proof of the conjecture}

In this section, we present a proof of  the conjecture of
Reiner, Tenner and Yong.
First, we establish the following relation on
$|{\mathrm{BSSYT}}(\lambda,k)|$ and $|{\mathrm{SYT}}(\lambda,k)|$
 for a balanced shape $\lambda$.

\begin{thm}\label{Gen-1}
For any positive integer $k$ and a balanced shape $\lambda$ with $r$ rows and $c$ columns, we have
\begin{equation}\label{3-1}
|{\mathrm{BSSYT}}(\lambda,k)|=\frac{krc}{(r+c)}
  |{\mathrm{SYT}}(\lambda,k)|.
\end{equation}
\end{thm}

As observed by Chan, Haddadan, Hopkins and Moci \cite{Chan-2},
 a rectangular staircase shape is a balanced shape.
  Moreover, a rectangular staircase
shape $\delta_d(b^a)$ has $a(d-1)$ rows and $b(d-1)$ columns.
Thus Theorem \ref{Gen-1} specializes to Conjecture \ref{Conj-1}.

To prove Theorem \ref{Gen-1},
we find the following representation of
 a barely set-valued tableau.

\begin{thm}\label{thm-subshape}
A barely set-valued tableau in ${\mathrm{BSSYT}}(\lambda,k)$ can be uniquely represented by a reverse
plane partition $P$ in ${\mathrm{RPP}}(\lambda,k)$ and an integer $i$ $(1\leq i\leq k)$ together with a  designated corner of the induced subshape $\alpha(P, i)$.
\end{thm}

\begin{proof}
Let $T$ be a barely set-valued   tableau in
${\mathrm{BSSYT}}(\lambda,k)$.
We aim to construct a reverse plane partition $P\in {\mathrm{RPP}}(\lambda,k)$, an integer $i$ $(1\leq i\leq k)$ and a corner $C$ in the induced subshape $\alpha(P, i)$.
For each entry in $T$, if it is in the $t$-th row, then
subtract it by
$t$. This results in a tableau $T'$ with every entry
not exceeding $k$ in which each row and each column
are weakly increasing.
Assume that $B$ is the  square of $T$ containing two entries, say, $a$ and $b$ with $a<b$, and assume that
$B$ is in the $r$-th row  of $T$.
By the above operation,  the  entries of $T'$
in the square $B$ are $a-r$ and $b-r$.
Define
$P$ to be the reverse plane partition in $\mathrm{RPP}(\lambda, k)$  obtained from $T'$ by deleting the entry $b-r$ in $B$.

We next proceed to determine the integer $i$ and the corner $C$ in the induced subshape $\alpha(P, i)$,
from which we can recover the deleted entry $b-r$ in the  tableau $T'$.
Notice that $r\leq a<b$. So we have $1\leq b-r\leq k$.
Set $i=b-r$.

We may choose the corner $C$ of  $\alpha(P, b-r)$ to be the square $B$. This is feasible because it can be shown that
the square $B$ is a corner of $\alpha(P, b-r)$.
Keep in mind that the
subshape $\alpha(P, b-r)$ consists of the squares of
$P$ occupied by the entries smaller than $b-r$.
Note that the entry in the square $B$ of $P$ is $a-r$.
Since $a-r<b-r$, the square $B$
must be a square of the subshape $\alpha(P, b-r)$. To verify that $B$ is a corner of $\alpha(P, b-r)$, we need to check that
if  $B'$ is a square  of $\lambda$  just to the right of $B$
or just below $B$, then $B'$ does not belong to $\alpha(P, b-r)$, or, equivalently, the entry of $P$ in  $B'$
is bigger than or equal to $b-r$. This is obvious
owing to  the construction of $P$. Thus $C$ is indeed a corner
of $\alpha(P, b-r)$.

To show that the above construction is reversible, we give a brief description of the reverse procedure. Given a reverse plane partition $P$ in $ \mathrm{RPP}(\lambda,k)$ together with an integer $1\leq i\leq k$ and a corner $C$ of $\alpha(P,i)$,  we shall recover a barely set-valued
tableau
$T$ in $\mathrm{BSSYT}(\lambda, k)$ as follows.
Let $T'$
be the tableau obtained from $P$ by joining the entry $i$ into the square $C$ so that the square $C$ has two entries.
 Increase each entry in $T'$ by $t$ if it is in the $t$-th row of $T'$. Let $T$ denote the resulting tableau.
It is easily
verified  that $T$ is a barely set-valued tableau in
$\mathrm{BSSYT}(\lambda, k)$. This completes the proof.
\end{proof}

Figure \ref{g-1} illustrates the  construction
of the representation of a barely set-valued
tableau $T$ in ${\mathrm{BSSYT}}(\lambda,k)$ with $\lambda=(4,4,2,1)$ and $k=2$, where the subshape $\alpha(P,2)$ is determined  by the lattice path in $\lambda$ drawn with thick line.

\begin{figure}[h]
\begin{center}
\begin{picture}(80,100)
\setlength{\unitlength}{1mm}
\put(-50,30){\line(1,0){20}}
\put(-50,25){\line(1,0){20}}
\put(-50,20){\line(1,0){20}}
\put(-50,15){\line(1,0){10}}
\put(-50,10){\line(1,0){5}}
\put(-50,10){\line(0,1){20}}
\put(-45,10){\line(0,1){20}}
\put(-40,15){\line(0,1){15}}
\put(-35,20){\line(0,1){10}}
\put(-30,20){\line(0,1){10}}
\put(-48.5,26.5){\small{$1$}}
\put(-43.5,26.5){\small{$1$}}
\put(-38.5,26.5){\small{$2$}}
\put(-33.5,26.5){\small{$2$}}
\put(-49.5,21.5){\small{$24$}}
\put(-43.5,21.5){\small{$4$}}
\put(-38.5,21.5){\small{$4$}}
\put(-33.5,21.5){\small{$4$}}
\put(-48.5,16.5){\small{$5$}}
\put(-43.5,16.5){\small{$5$}}
\put(-48.5,11.5){\small{$6$}}

\put(-26,18.5){{$\xlongleftrightarrow{\quad}$}}

\put(-10,30){\line(1,0){20}}
\put(-10,25){\line(1,0){20}}
\put(-10,20){\line(1,0){20}}
\put(-10,15){\line(1,0){10}}
\put(-10,10){\line(1,0){5}}
\put(-10,10){\line(0,1){20}}
\put(-5,10){\line(0,1){20}}
\put(0,15){\line(0,1){15}}
\put(5,20){\line(0,1){10}}
\put(10,20){\line(0,1){10}}
\put(-8.5,26.5){\small{$0$}}
\put(-3.5,26.5){\small{$0$}}
\put(1.5,26.5){\small{$1$}}
\put(6.5,26.5){\small{$1$}}
\put(-9.5,21.5){\small{$0$}}
\put(-7.5,21.5){\small{$2$}}
\put(-3.5,21.5){\small{$2$}}
\put(1.5,21.5){\small{$2$}}
\put(6.5,21.5){\small{$2$}}
\put(-8.5,16.5){\small{$2$}}
\put(-3.5,16.5){\small{$2$}}
\put(-8.5,11.5){\small{$2
$}}

\put(14,18.5){{$\xlongleftrightarrow{\quad}$}}

\put(34,30){\line(1,0){20}}
\put(34,25){\line(1,0){20}}
\put(34,20){\line(1,0){20}}
\put(34,15){\line(1,0){10}}
\put(34,10){\line(1,0){5}}
\put(34,10){\line(0,1){20}}
\put(39,10){\line(0,1){20}}
\put(44,15){\line(0,1){15}}
\put(49,20){\line(0,1){10}}
\put(54,20){\line(0,1){10}}

\put(30,6){\line(0,1){28}}
\put(30,6){\line(1,0){3}}
\put(30,34){\line(1,0){3}}

\put(35.5,26.5){\small{$0$}}
\put(40.5,26.5){\small{$0$}}
\put(45.5,26.5){\small{$1$}}
\put(50.5,26.5){\small{$1$}}
\put(35.5,21.5){\small{$0$}}
\put(40.5,21.5){\small{$2$}}
\put(45.5,21.5){\small{$2$}}
\put(50.5,21.5){\small{$2$}}
\put(35.5,16.5){\small{$2$}}
\put(40.5,16.5){\small{$2$}}
\put(35.5,11.5){\small{$2$}}

\put(60,30){\line(1,0){20}}
\put(60,25){\line(1,0){20}}
\put(60,20){\line(1,0){20}}
\put(60,15){\line(1,0){10}}
\put(60,10){\line(1,0){5}}
\put(60,10){\line(0,1){20}}
\put(65,10){\line(0,1){20}}
\put(70,15){\line(0,1){15}}
\put(75,20){\line(0,1){10}}
\put(80,20){\line(0,1){10}}

\put(84,6){\line(0,1){28}}
\put(84,6){\line(-1,0){3}}
\put(84,34){\line(-1,0){3}}
\put(61,21.2){\footnotesize{$C$}}
\linethickness{2pt}
\put(60,10){\line(0,1){10.3}}
\put(60,20){\line(1,0){5.3}}
\put(65,20){\line(0,1){5.3}}
\put(65,25){\line(1,0){15.3}}
\put(80,25){\line(0,1){5.3}}

\put(-40,0){$T$}
\put(0,0){$T'$}
\put(45,0){$(P,\alpha(P,2),C)$}
\end{picture}
\caption{An illustration of the representation in Theorem \ref{thm-subshape}.}\label{g-1}
\end{center}
\end{figure}

In the spirit of Theorem \ref{thm-subshape},
we have an alternative representation of
 a barely set-valued tableau involving a designated proper outside corner.

\begin{thm}\label{thm-subshape'}
A barely set-valued tableau in ${\mathrm{BSSYT}}(\lambda,k)$ can be uniquely represented by a reverse
plane partition $Q$ in ${\mathrm{RPP}}(\lambda,k)$ and an integer $j$ $(1\leq j\leq k)$ together with a  designated
proper outside corner
 of the induced subshape $\alpha(Q, j)$.
\end{thm}

\begin{proof}
The proof is similar to that of
Theorem \ref{thm-subshape}, and so we only give a description of the construction from a  barely set-valued tableau $T$ in ${\mathrm{BSSYT}}(\lambda,k)$ to a reverse
plane partition $Q$ in ${\mathrm{RPP}}(\lambda,k)$ and an integer $j$ $(1\leq j\leq k)$ together with a  designated
proper outside corner $C'$ of $\alpha(Q, j)$.

Let $T'$ be the tableau as constructed in the proof of Theorem
\ref{thm-subshape}.
Define $Q$ to be the reverse plane partition in $\mathrm{RPP}(\lambda, k)$  obtained from $T'$ by deleting the entry $a-r$ in $B$.
Set $j=a-r+1$. It can be verified that $B$
is a proper outside corner of $\alpha(Q, a-r+1)$.
Then choose $C'$ to be the proper outside corner $B$. This completes the proof.
\end{proof}

 Figure \ref{g-2} is an illustration of Theorem
\ref{thm-subshape'}, where $T$ is    a barely set-valued tableau in $ {\mathrm{BSSYT}}(\lambda,k)$ with $\lambda=(4,4,2,1)$
and $k=2$.

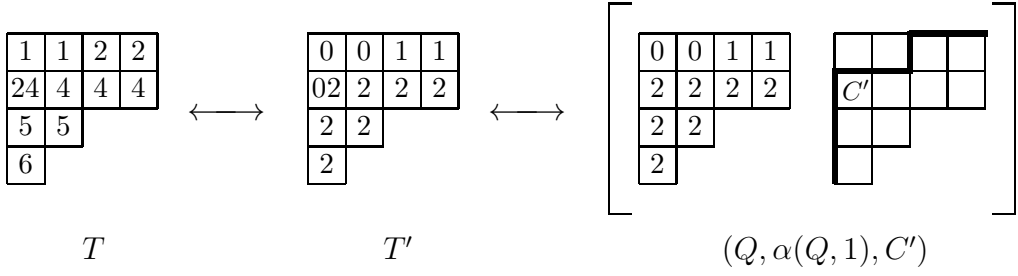
\begin{figure}[h]
\begin{center}
\begin{picture}(80,100)
\setlength{\unitlength}{1mm}
\put(-50,30){\line(1,0){20}}
\put(-50,25){\line(1,0){20}}
\put(-50,20){\line(1,0){20}}
\put(-50,15){\line(1,0){10}}
\put(-50,10){\line(1,0){5}}
\put(-50,10){\line(0,1){20}}
\put(-45,10){\line(0,1){20}}
\put(-40,15){\line(0,1){15}}
\put(-35,20){\line(0,1){10}}
\put(-30,20){\line(0,1){10}}
\put(-48.5,26.5){\small{$1$}}
\put(-43.5,26.5){\small{$1$}}
\put(-38.5,26.5){\small{$2$}}
\put(-33.5,26.5){\small{$2$}}
\put(-49.5,21.5){\small{$24$}}
\put(-43.5,21.5){\small{$4$}}
\put(-38.5,21.5){\small{$4$}}
\put(-33.5,21.5){\small{$4$}}
\put(-48.5,16.5){\small{$5$}}
\put(-43.5,16.5){\small{$5$}}
\put(-48.5,11.5){\small{$6$}}

\put(-26,18.5){{$\xlongleftrightarrow{\quad}$}}

\put(-10,30){\line(1,0){20}}
\put(-10,25){\line(1,0){20}}
\put(-10,20){\line(1,0){20}}
\put(-10,15){\line(1,0){10}}
\put(-10,10){\line(1,0){5}}
\put(-10,10){\line(0,1){20}}
\put(-5,10){\line(0,1){20}}
\put(0,15){\line(0,1){15}}
\put(5,20){\line(0,1){10}}
\put(10,20){\line(0,1){10}}
\put(-8.5,26.5){\small{$0$}}
\put(-3.5,26.5){\small{$0$}}
\put(1.5,26.5){\small{$1$}}
\put(6.5,26.5){\small{$1$}}
\put(-9.5,21.5){\small{$0$}}
\put(-7.5,21.5){\small{$2$}}
\put(-3.5,21.5){\small{$2$}}
\put(1.5,21.5){\small{$2$}}
\put(6.5,21.5){\small{$2$}}
\put(-8.5,16.5){\small{$2$}}
\put(-3.5,16.5){\small{$2$}}
\put(-8.5,11.5){\small{$2
$}}

\put(14,18.5){{$\xlongleftrightarrow{\quad}$}}

\put(34,30){\line(1,0){20}}
\put(34,25){\line(1,0){20}}
\put(34,20){\line(1,0){20}}
\put(34,15){\line(1,0){10}}
\put(34,10){\line(1,0){5}}
\put(34,10){\line(0,1){20}}
\put(39,10){\line(0,1){20}}
\put(44,15){\line(0,1){15}}
\put(49,20){\line(0,1){10}}
\put(54,20){\line(0,1){10}}

\put(30,6){\line(0,1){28}}
\put(30,6){\line(1,0){3}}
\put(30,34){\line(1,0){3}}

\put(35.5,26.5){\small{$0$}}
\put(40.5,26.5){\small{$0$}}
\put(45.5,26.5){\small{$1$}}
\put(50.5,26.5){\small{$1$}}
\put(35.5,21.5){\small{$2$}}
\put(40.5,21.5){\small{$2$}}
\put(45.5,21.5){\small{$2$}}
\put(50.5,21.5){\small{$2$}}
\put(35.5,16.5){\small{$2$}}
\put(40.5,16.5){\small{$2$}}
\put(35.5,11.5){\small{$2$}}

\put(60,30){\line(1,0){20}}
\put(60,25){\line(1,0){20}}
\put(60,20){\line(1,0){20}}
\put(60,15){\line(1,0){10}}
\put(60,10){\line(1,0){5}}
\put(60,10){\line(0,1){20}}
\put(65,10){\line(0,1){20}}
\put(70,15){\line(0,1){15}}
\put(75,20){\line(0,1){10}}
\put(80,20){\line(0,1){10}}

\put(84,6){\line(0,1){28}}
\put(84,6){\line(-1,0){3}}
\put(84,34){\line(-1,0){3}}
\put(61,21.2){\footnotesize{$C'$}}
\linethickness{2pt}
\put(60,10){\line(0,1){15.3}}
\put(60,25){\line(1,0){10.3}}
\put(70,25){\line(0,1){5.3}}
\put(70,30){\line(1,0){10.3}}

\put(-40,0){$T$}
\put(0,0){$T'$}
\put(45,0){$(Q,\alpha(Q,1),C')$}
\end{picture}
\caption{An alternative representation.
}\label{g-2}
\end{center}
\end{figure}

We are now ready to complete the proof of Theorem \ref{Gen-1} based on the above two representations of a
barely set-valued tableau and the formula in Theorem \ref{JD1}.

\begin{proof}[Proof of Theorem \ref{Gen-1}]
Recall  that the
expected jaggedness of a subshape of $\lambda$ under the
weak distribution equals
\begin{equation}\label{ZZ}
\frac{\sum_{\mu}|\{(P,i)\in Q(\lambda, k)\,|\,\alpha(P,i)=\mu\}|\mathrm{jag}(\mu)}{|Q(\lambda, k)|},
\end{equation}
where $\mu$ ranges over subshapes of $\lambda$.
To compute the numerator of \eqref{ZZ}, note that
\begin{equation*}
\sum_{\mu}|\{(P,i)\in Q(\lambda, k)\,|\,\alpha(P,i)=\mu\}|\mathrm{jag}(\mu)=\sum_{(P,i)\in Q(\lambda,k)}\mathrm{jag}(\alpha(P,i)).
\end{equation*}
Let $C(P,i)$ denote the number of
corners in the subshape $\alpha(P,i)$, and let
$C'(P,i)$ denote the number
of proper outside corners of  $\alpha(P,i)$.
Then we have
\[\mathrm{jag}(\alpha(P,i))=C(P,i)+C'(P,i).\]
Recalling that
\[Q(\lambda,k)=\{(P,i)\,|\, P\in \mathrm{RPP}(\lambda, k), 1 \leq i\leq k\},\]
we get
\begin{align} \label{vvv}
\sum_{(P,i)\in Q(\lambda,k)}\mathrm{jag}(\alpha(P,i))&=\sum_{P\in {\mathrm{ RPP}}(\lambda,k)}\sum_{i=1}^k\,\mathrm{jag}(\alpha(P,i))\nonumber\\[5pt]
&=\sum_{P\in {\mathrm{ RPP}}(\lambda,k)}\sum_{i=1}^k C(P,i) +
\sum_{P\in {\mathrm{ RPP}}(\lambda,k)}\sum_{i=1}^k C'(P,i).
\end{align}
By Theorem \ref{thm-subshape} and Theorem  \ref{thm-subshape'},
both the first double sum  and the second double sum  in \eqref{vvv} are   equal to $|{\mathrm{BSSYT}}(\lambda,k)|$. It
follows that
\begin{equation}\label{ZZZ}
\sum_{\mu}|\{(P,i)\in Q(\lambda, k)\,|\,\alpha(P,i)=\mu\}|\mathrm{jag}(\mu)=  2|{\mathrm{BSSYT}}(\lambda,k)|.
\end{equation}

As to the denominator of  \eqref{ZZ},
we notice that \[ |Q(\lambda, k)|=k|{\mathrm{ RPP}}(\lambda,k)| . \]
On the other hand, there is an obvious bijection between the set ${\mathrm{ RPP}}(\lambda,k)$ and the set ${\mathrm{ SYT}}(\lambda,k)$. Given  a reverse plane partition
$P\in {\mathrm{ RPP}}(\lambda,k)$, one can construct a
semistandard Young tableau in ${\mathrm{ SYT}}(\lambda,k)$  from $P$ by increasing
each entry in the $t$-th row of $P$ by $t$.   Therefore,
\begin{equation}\label{ZZ2}
|Q(\lambda, k)|=k|{\mathrm{ SYT}}(\lambda,k)|.
 \end{equation}
Substituting  \eqref{ZZZ} and \eqref{ZZ2} into \eqref{ZZ},  the expected jaggedness in
\eqref{ZZ} can be rewritten as
\begin{equation*}
\frac{2|{\mathrm{BSSYT}}(\lambda,k)|}{k|{\mathrm{ SYT}}(\lambda,k)|},
\end{equation*}
which, together with Theorem \ref{JD1}, yields
\begin{equation*}
\frac{2|{\mathrm{BSSYT}}(\lambda,k)|}
{k|{\mathrm{SYT}}(\lambda,k)|}=\frac{2rc}{r+c}.
\end{equation*}
 This confirms \eqref{3-1}, and hence the proof is complete.
\end{proof}

We conclude this paper with a formula on the polynomial   $FK(w,\ell)$ with respect to a
dominant permutation $w$ corresponding to a balanced shape.
The proof is based on  Theorem \ref{Gen-1} and a relation on
$FK(w,\ell)$ established by Reiner, Tenner and Yong \cite{Reiner}.
Restricting to a dominant permutation corresponding to a rectangular staircase shape, this formula
reduces to Conjecture \ref{Conj-1-e}. Bear in mind that
$FK(w,\ell)$ is a polynomial in $x$ of degree $\ell$. For the reason that the proof of the following theorem involves the evaluation of the polynomials $FK(w,\ell)$
at $x= 1, 2, \ldots$, we shall write $FK_{w,\ell}(x)$ for $FK(w,\ell)$.

\begin{thm}\label{fi}
Let $w$ be a dominant permutation whose Lehmer code  is a balanced shape $\lambda$ with $r$ rows and $c$ columns, and let  $\ell=\ell(w)$. Then we have
\begin{equation}\label{ffff}
\frac{FK(w,\ell+1)}{FK(w,\ell)}
= \binom{ \ell+1 }{2}\left( \frac{ 2xrc }{\ell(r+c)} + 1\right).
\end{equation}
\end{thm}

\begin{proof}
In  the proof of \cite[Corollary 6.11]{Reiner}, Reiner,
Tenner and Yong established  the following relation
for any dominant permutation $w$ and any positive integer $k$:
\begin{equation}\label{finn}
\frac{FK_{w,\ell+1}(k)}{FK_{w,\ell}(k)}
 =\binom{\ell+1}{2} + (\ell+1)   \frac{|{\mathrm{BSSYT}}(\lambda,k)|}{|{\mathrm{SYT}}(\lambda,k)|}.
\end{equation}
Substituting \eqref{3-1} into    \eqref{finn}, we obtain that
\begin{equation*}
\frac{FK_{w,\ell+1}(k)}{FK_{w,\ell}(k)}=
\binom{\ell+1}{2}+(\ell+1)\frac{krc}{r+c}
=\binom{ \ell+1 }{2}\left( \frac{ 2krc }{\ell(r+c)} + 1\right),
\end{equation*}
that is to say that \eqref{ffff} holds for any positive integer $k$. Since \eqref{ffff} can be recast as a relation on polynomials, it holds for $FK_{w,\ell}(x)$. This completes
the proof.
\end{proof}

Notice that by \eqref{finn}, it is clear  that   Theorem \ref{fi} is equivalent to Theorem \ref{Gen-1}.

\vskip 3mm \noindent {\bf Acknowledgments.}
This work was supported
by the 973 Project, and the National Science Foundation of
China.

\end{CJK*}

\begin{thebibliography}{99}

\bibitem{Buch}
A. Buch, A Littlewood-Richardson rule for the K-theory of Grassmannians, Acta Math. 189 (2002), 37--78.

\bibitem{Chan-1}
M. Chan, A. L\'opez Mart\'in, N. Pflueger and M. Teixidor i Bigas,
Genera of Brill-Noether
curves and staircase paths in Young tableaux, Trans. Amer. Math. Soc.
370 (2018), 3405--3439.

\bibitem{Chan-2}
M. Chan, S. Haddadan, S. Hopkins and L. Moci, The expected jaggedness of order ideals,
Forum Math. Sigma 5 (2017), e9, 27pp.

\bibitem{Fomin}
S. Fomin and A.N. Kirillov, Reduced words and plane partitions, J. Algebraic Combin. 6 (1997), 311--319.

\bibitem{Knutson}
A. Knutson, E. Miller and A. Yong, Gr\"obner geometry of vertex decompositions and of flagged tableaux, J. Reine Angew. Math. 630 (2009), 1--31.

\bibitem{Macdonald}
I.G. Macdonald, Notes on Schubert polynomials, Laboratoire de combinatoire et d'informatique math\'ematique
(LACIM), Universit\'e du Qu\'ebec $\grave{a}$ Montr\'eal, Montr\'eal, 1991.

\bibitem{Pak}
I. Pak, Partition bijections, a survey, Ramanujan J. 12 (2006), 5--75.

\bibitem{Proctor}
R.A. Proctor, Odd symplectic groups, Invent. Math. 92 (1988), 307--332.

\bibitem{Reiner}
V. Reiner, B.E. Tenner and A. Yong, Poset edge densities, nearly reduced words, and
barely set-valued tableaux, J. Combin. Theory Ser. A 158 (2018), 66--125.

\bibitem{Stanley}
  R.P. Stanley, Enumerative Combinatorics, Vol. 2, Cambridge Studies in Advanced Mathematics,   Cambridge University Press, Cambridge, 1999.

\bibitem{Stanley1}
R.P. Stanley, On the number of reduced decompositions of elements of Coxeter groups, European J. Combin. 5 (1984), 359--372.
\end{thebibliography}
\end{document}